\documentclass[11pt, oneside]{article}   	% use "amsart" instead of "article" for AMSLaTeX format
\usepackage[scale=0.70]{geometry}                		% See geometry.pdf to learn the layout options. There are lots.
\usepackage{graphicx}				% Use pdf, png, jpg, or epsÂ§ with pdflatex; use eps in DVI mode
\usepackage[dvipsnames]{xcolor}
\usepackage[linktocpage]{hyperref}
\hypersetup{
    colorlinks=true,
    linkcolor=black,
    hypertexnames=false,
    filecolor=black,      
    urlcolor=black,
    citecolor=black
}
\usepackage{amssymb}
\usepackage{amsmath}
\usepackage{amsthm}
\usepackage{enumitem}
\usepackage{authblk}

\newtheorem{theorem}{Theorem}%[section]
\newtheorem{lemma}[theorem]{Lemma}
\newtheorem{proposition}[theorem]{Proposition}

\newtheorem{conjecture}[theorem]{Conjecture}

\theoremstyle{definition}
\newtheorem{defn}[theorem]{Definition}

\newtheorem{example}[theorem]{Example}
\newtheorem{remark}[theorem]{Remark}

\def\B{\mathbf B}
\def\C{\mathbb C}
\def\H{\mathbb H}
\def\P{\mathbb{P}}
\def\N{\mathbb N}
\def\R{\mathbb R}

\def\<{\langle}
\def\>{\rangle}
\def\D{\partial}

\def\ol{\overline}
\def\Ga{\Gamma}
\def\Om{\Omega}
\def\ga{\gamma}

\def\La{\Lambda}
\def\Aut{\operatorname{Aut}}
\def\inter{\operatorname{int}}
\def\ra{\rightarrow}

\DeclareMathOperator{\Hess}{Hess}
\DeclareMathOperator{\isom}{Isom}

\DeclareMathOperator{\vol}{Vol}

\title{A note on complex-hyperbolic Kleinian groups}
\author{Subhadip Dey\thanks{ Dept. of Math., UC Davis,
One Shields Ave, Davis, CA 95616, email: \texttt{sdey@math.ucdavis.edu}}\ \qquad
Michael Kapovich\thanks{
Dept. of Math., UC Davis,
One Shields Ave, Davis, CA 95616, email: \texttt{kapovich@math.ucdavis.edu}}}
\date{May 2, 2020}							% Activate to display a given date or no date

\begin{document}
\maketitle

\begin{abstract}
 Let $\Ga$ be a discrete group of isometries acting on the complex hyperbolic $n$-space $\H^n_\C$. In this note, we prove that if $\Ga$ is convex-cocompact, torsion-free, and the critical exponent $\delta(\Ga)$ is strictly lesser than $2$, then the complex manifold $\H^n_\C/\Ga$ is Stein. We also discuss several related conjectures.
 \end{abstract}

The theory of complex hyperbolic manifolds and complex-hyperbolic Kleinian groups (i.e. discrete holomorphic isometry groups of complex hyperbolic spaces $\H^n_\C$) is a rich mixture of Riemannian and complex geometry, topology, dynamics, symplectic geometry and complex analysis. The purpose of this note is to discuss 
interactions of the theory of complex-hyperbolic Kleinian groups and the function theory of complex-hyperbolic manifolds. 
Let $\Gamma$ be a discrete group of isometries acting on the complex-hyperbolic $n$-space, $\H^n_\C$, the unit ball $\B^n\subset \C^n$ equipped with the Bergmann metric.
A fundamental numerical invariant associated with $\Ga$ is the {\em critical exponent} $\delta(\Ga)$ of $\Gamma$, defined by
\[
\delta(\Ga) = \inf\left\{ s : \sum_{\gamma\in\Ga} e^{-s\cdot d(x,\gamma x)} <\infty \right\},
\]
where $x\in \H^n_\C$ is any\footnote{$\delta(\Ga)$ does not depend on the choice of $x\in \H^n_\C$.} point. The critical exponent measures the rate of exponential growth the $\Ga$-orbit $\Ga x\subset  \H^n_\C$; it also equals the Haussdorff dimension of the conical limit set of $\Ga$, see \cite{Corlette} and \cite{CO}. 

Our main result is: 

\begin{theorem}\label{thm:stein}
Suppose that $\Ga< \Aut(\B^n)$ is a convex-cocompact, torsion-free discrete subgroup satisfying $\delta(\Ga)<2$. Then $M_\Ga = \B^n/\Ga$ is Stein.   
\end{theorem}

The condition on the critical exponent in the above theorem is sharp since, for a complex Fuchsian subgroup $\Ga< \Aut(\B^n)$, $\delta(\Ga) = 2$, but the quotient $M_\Ga = \B^n/\Ga$ is non-Stein because the convex core of $M_\Ga$ is a complex curve, see Example \ref{ex:cFuch}.
On the other hand, if $\Ga$ is a torsion-free real Fuchsian subgroup or a small deformation of such (see Example \ref{ex:rFuch}), then $\Ga$ satisfies the condition of the above theorem.

The main ingredients in the proof of Theorem \ref{thm:stein} are Proposition \ref{prop:stein} and Theorem \ref{thm:no_subv}. The condition ``convex-cocompact'' is only used in Proposition \ref{prop:stein}, whereas Theorem \ref{thm:no_subv} holds for any torsion-free discrete subgroup $\Ga<\Aut(\B^n)$ satisfying $\delta(\Ga)<2$. 

%We conjecture that:

\begin{conjecture}
Theorem \ref{thm:stein} holds  if we omit the ``convex-cocompact'' assumption on $\Ga$. 
\end{conjecture} 

In section \ref{sec:remarks} we discuss other conjectural generalizations of Theorem \ref{thm:stein} and supporting results. 

\medskip 

\noindent
{\bf Acknowledgement.} The second author was partly supported by the NSF grant  DMS-16-04241.

\section{Preliminaries}

In this section, we recall some definitions and basic facts about the $n$-dimensional complex hyperbolic space, we refer to \cite{Goldman, Kapovich} for details.

Consider the $n$-dimensional complex vector space $\C^{n+1}$ equipped with the pseudo-hermitian bilinear form
\begin{equation}\label{eqn:hermitianform}
\<z, w\>=-z_0 \bar{w}_0 + \sum_{k=1}^n z_k \bar{w}_k  
\end{equation}
and define the quadratic form $q(z)$ of signature $(n,1)$ by $q(z):= \<z, z\>$.
Then $q$ defines the {\em negative light cone} $V_-:= \{z: q(z)<0\} \subset \C^{n+1}$.
The projection of $V_-$ in the projectivization of $\C^{n+1}$, $\P^{n}$, is an open ball which we denote by $\B^n$.

The tangent space $T_{[z]} \P^n$ is naturally identified with $z^\perp$, the orthogonal complement of $\C z$ in $V$, taken with respect to $\<\cdot, \cdot\>$. 
If $z\in V_-$, then the restriction of $q$ to $z^\perp$ is positive-definite, hence, $\<\cdot, \cdot\>$ 
project to a hermitian metric $h$ (also denoted $\<\cdot, \cdot\>_h$) on $\B^n$. 
The {\em complex hyperbolic $n$-space} $\H^n_{\C}$ is $\B^n$ equipped with the hermitian metric $h$.
The boundary $\partial\B^n$ of $\B^n$ in $\P^n$ gives a natural compactification of $\B^n$.

In this note, we usually denote the complex hyperbolic $n$-space by $\B^n$.
The real part of the hermitian metric $h$ defines a Riemannian metric $g$ on $\B^n$.
The sectional curvature of $g$ varies between $-4$ and $-1$.
We denote the distance function on $\B^n$ by $d$. The distance function satisfies
\begin{equation}\label{eqn:dist}
\cosh^2(d(0, z))= (1-|z|^2)^{-1}.
\end{equation}

A real linear subspace $W\subset \C^{n+1}$ is said to be {\em totally real} with respect to the form (\ref{eqn:hermitianform}) if for any two vectors 
$z, w\in W$,  $\<z, w\>\in \R$. Such a subspace is automatically totally real in the usual sense: $J W\cap W=\{0\}$, where $J$ is the almost complex structure on $V$. 
{\em (Real) geodesics} in $\B^n$ are projections of totally real indefinite (with respect to $q$) 2-planes in $\C^{n+1}$ (intersected with $V_-$). For instance, geodesics through the origin $0\in \B^n$ are Euclidean line segments in $\B^n$. 
More generally, totally-geodesic real subspaces in $\B^n$ are projections of totally real indefinite subspaces in $\C^{n+1}$ (intersected with $V_-$). They are isometric to the real hyperbolic space $\H^n_\R$ of constant sectional curvature $-1$.

{\em Complex geodesics} in $\B^n$ are projections of indefinite complex 2-planes.
Complex geodesics are isometric to the unit disk with the hermitian metric 
$$
\frac{dz d\bar z}{(1-|z|^2)^2},
$$ 
which has constant sectional curvature $-4$.
More generally, $k$-dimensional complex hyperbolic  subspaces $\H^k_\C$ 
in $\B^n$ are projections of indefinite complex $(k+1)$-dimensional subspaces 
(intersected with $V_-$). 

All complete totally-geodesic submanifolds in $\H^n_\C$ are either real or complex hyperbolic subspaces. 

The group $\mathrm{U}(n,1)\cong \mathrm{U}(q)$ of (complex) automorphisms of the form $q$ projects to the group $\Aut(\B^n) \cong \mathrm{PU}(n,1)$ of complex (biholomorphic, isometric) 
automorphisms of $\B^n$. The group $\Aut(\B^n)$ is linear, its matrix representation is given, for instance, by the adjoint representation, which is faithful since $\Aut(\B^n)$ has trivial center. 

A discrete subgroup $\Ga$ of $\Aut(\B^n)$ is called a {\em complex-hyperbolic Kleinian group}. 
The accumulation set of an(y) orbit $\Ga x$ in $\partial\B^n$ is called the {\em limit set} of $\Ga$ and denoted by $\Lambda(\Gamma)$.
The complement of $\Lambda(\Gamma)$ in $\partial\B^n$ is called the {\em domain of discontinuity} of $\Ga$ and denoted by $\Omega(\Ga)$.
The group $\Ga$ acts properly discontinuously on $\B^n \cup \Omega(\Ga)$. 

For a (torsion-free) complex-hyperbolic Kleinian group $\Ga$, the quotient $\B^n/\Ga$ is a Riemannian orbifold (manifold) equipped with push-forward of the Riemannian metric of $\B^n$.
We reserve the notation $M_\Ga$ to denote this quotient.
The {\em convex core} of $M_\Ga$ is the the projection of the closed convex hull of $\Lambda(\Ga)$ in $\B^n$.
The subgroup $\Ga$ is called {\em convex-cocompact} if the convex core of $M_\Ga$ is a nonempty compact subset. Equivalently (see \cite{Bowditch}), 
$\ol{M}_\Ga=(\B^n \cup \Omega(\Ga))/\Ga$ is compact. 

Below are two  interesting examples of convex-cocompact complex-hyperbolic Kleinian groups which will also serve as illustrations our results.

\begin{example}[Real Fuchsian subgroups]\label{ex:rFuch}
Let $\H^2_\R\subset \B^n$ be a totally real hyperbolic plane. 
This inclusion is induced by an embedding $\rho:\isom(\H^2_\R) = \mathrm{PSL}(2,\R) \ra \Aut(\B^n)$ whose image preserves $\H^2_\R$. 
Let $\Ga'< \isom(\H^2_\R)$ be a uniform lattice. Then $\Ga = \rho(\Ga')$ 
preserves  $\H^2_\R$ and acts on it cocompactly. 
Such subgroups $\Ga<\Aut(\B^n)$ will be called {\em real Fuchsian subgroups}.
The compact surface-orbifold $\Sigma = \H^2_\R/{\Ga}$ is the convex core of $M_{\Ga}$. The critical exponent $\delta(\Ga)$ is 1.

Let $\Ga_t$, $t\ge 0$, be a continuous family of deformations  of $\Ga_0 = \Ga$ in $\Aut(\B^n)$  such that $\Ga_t$'s, for $t>0$, are convex-cocompact but not real Fuchsian.
Such deformation exist as long as $\Ga_t$ is, say, torsion-free, see e.g. \cite{Weil}.
The groups $\Ga_t$, $t>0$, are called {\em real quasi-Fuchsian subgroups}.
The critical exponents of such subgroups are strictly greater than $1$.
\end{example}

\begin{example}[Complex Fuchsian subgroups]\label{ex:cFuch}
Let $\Ga'$ be a cocompact subgroup of $\mathrm{SU}(1,1)$, the identity component isometry group of the real-hyperbolic plane (modulo $\mathbb{Z}_2$) and let $\mathrm{SU}(1,1) \ra \mathrm{SU}(n,1)$ be any embedding.
Note that $\mathrm{SU}(n,1)$ modulo center (isomorphic to $\mathbb{Z}_{n+1}$) is isomorphic to $\mathrm{PU}(n,1)$.
By taking compositions, we get a representation $\rho:\Ga'\ra\mathrm{PU}(n,1)$.
Then $\Ga:= \rho(\Ga')$ leaves a complex geodesic invariant in $\B^n$.
Such subgroups $\Ga$ will be called {\em complex Fuchsian subgroups}.
In this case, $\mathrm{core}(M_{\Ga}) = \H^1_\C/\Ga$ is  a compact complex curve in $M_{\Ga}$ where $\H^1_\C$ is the $\Ga$-invariant complex geodesic.
The critical exponent $\delta(\Ga)$ is $2$.
\end{example}

\section{Generalities on complex manifolds} 
By a {\em complex manifold with boundary} $M$, we mean a smooth manifold with (possibly empty) boundary $\partial M$ 
such that $\inter(M)$ is equipped with a complex structure and that there exists a smooth embedding  
$f: M\to X$ to an equidimensional complex manifold $X$,  biholomorphic on $\inter(M)$.  
A holomorphic function on $M$ is a smooth function which admits a holomorphic extension to a neighborhood of $M$ in $X$. 

Let $X$ be a complex manifold and $Y\subset X$ is a codimension $0$ smooth submanifold with boundary in $X$. The submanifold 
$Y$ is said to be {\em strictly Levi-convex} if  every boundary point of $Y$ 
admits a neighborhood $U$ in $X$ such that the submanifold with boundary $Y\cap U$ can be written  as
$$
\{\phi\le 0\},
$$
for some smooth submersion $\phi: U\to \R$ satisfying
$
\Hess(\phi) >0,  
$
where $\Hess(\phi)$ is the holomorphic Hessian:
$$
\left( \frac{\partial^2 \phi}{\partial \bar{z}_i\partial z_j}\right).  
$$

\begin{defn}
A {\em strongly pseudoconvex manifold} $M$ is a complex manifold with boundary 
which admits a strictly Levi-convex holomorphic embedding in an equidimensional complex manifold. 
\end{defn} 

\begin{defn}
An open complex manifold $Z$ is called {\em holomorphically convex} if for every discrete closed subset $A\subset Z$ there exists a holomorphic function $Z\to \C$ which is proper on $A$. 
\end{defn}

Alternatively,\footnote{and this is the standard definition} one can define holomorphically convex manifolds as follows: For a compact $K$ 
in a complex manifold $M$, the {\em holomorphic convex hull} $\hat{K}_M$ of $K$ in $M$ is 
$$
\hat{K}_M= \{z\in M: |f(z)|\le \sup_{w\in K} |f(w)|, \forall f\in \mathcal{O}_M\}.  
$$
In the above, $\mathcal{O}_M$ denotes the ring of holomorphic functions on $M$.
Then $M$ is holomorphically convex iff for every compact $K\subset M$, the hull $\hat{K}_M$ is also compact. 

\begin{theorem}
[Grauert \cite{G}] The interior of every compact strongly pseudoconvex manifold $M$ is holomorphically convex. 
\end{theorem}

\begin{defn}
A complex manifold $M$ is called {\em Stein} if it admits a proper holomorphic embedding in $\C^n$ for some $n$. 
\end{defn}

Equivalently, $M$ is Stein iff it is holomorphically convex and {\em holomorphically separable}: That is, for every distinct points $x,y\in M$, there exists a holomorphic function $f: M\ra\C$ such that $f(x) \ne f(y)$. We will use: 

\begin{theorem}[Rossi \cite{R2}, Corollary on page 20]\label{thm:Rossi}
 If  a compact complex manifold $M$ is strongly pseudoconvex and contains no compact complex subvarieties of positive dimension, then $\inter(M)$ is Stein. 
\end{theorem}

We now discuss strong quasiconvexity and Stein property  in the context of complex-hyperbolic manifolds. 
A classical example of a complex submanifold with Levi-convex boundary is a closed round ball $\ol{\B}^n$ in $\C^n$. Suppose that $\Ga< \Aut(\B^n)$ is a 
discrete torsion-free subgroup of the group of holomorphic automorphisms of $\B^n$ with (nonempty) domain of discontinuity 
$\Omega = \Omega(\Ga)\subset \partial \B^n$. The quotient 
$$
\ol{M}_\Ga= (\B^n \cup \Om)/\Ga
$$
is a smooth manifold with boundary.

\begin{lemma}\label{lem:spc}
$\ol{M}_\Ga$ is strongly pseudoconvex. 
\end{lemma}

\begin{proof}
 We let $T_\La$ denote the union of all projective hyperplanes in $P_\C^n$ tangent to $\D \B^n$ at points of $\La$, the limit set of $\Ga$. Let 
$\widehat{\Om}$ denote the connected component of $P_\C^n \setminus T_\La$ containing $\B^n$. It is clear that $\B^n\cup \Om\subset \widehat\Om$ is strictly Levi-convex. By the construction, $\Ga$ preserves $\widehat\Om$. 
It is proven in \cite[Thm. 7.5.3]{CNS} that the action of $\Ga$ on $\widehat\Om$
 is properly discontinuous. Hence, $X:=\widehat\Om/\Ga$ is a complex manifold containing  $\ol{M}_\Ga$ as a strictly Levi-convex submanifold with boundary. 
\end{proof} 
 
Specializing to the case when $\ol{M}_\Ga$ is compact, i.e. $\Ga$ is convex-cocompact, we obtain:

\begin{proposition}\label{prop:stein}
Suppose that $\Ga$ is torsion-free, convex-cocompact and $n>1$. Then:
\begin{enumerate}[nolistsep]
 \item $\partial \ol{M}_\Ga$ is connected. 
 \item If $\inter(\ol{M}_\Ga) = M_\Ga$ contains no compact complex subvarieties of positive dimension, then $M_\Ga$ is Stein.  
\end{enumerate}
\end{proposition}

For example, as it was observed in \cite{BS}, the quotient-manifold $\B^2/\Ga$ of a real-Fuchsian subgroup $\Ga< \Aut(\B^2)$ is Stein while the quotient-manifold 
 of a complex-Fuchsian subgroup $\Ga<  \Aut(\B^2)$ is non-Stein. 

\section{Proof of Theorem \ref{thm:stein}}

In this section, we construct certain {plurisubharmonic} functions on $M_\Ga$, for each finitely generated, discrete subgroup $\Gamma < \Aut(\B^n)$ satisfying $\delta(\Gamma) < 2$.
We use these functions to show that $M_\Ga$ has no compact subvarieties of positive dimension.
At the end of this section, we  prove the main result of this paper.

Let $X$ be a complex manifold.
Recall that a continuous function $f:X\rightarrow \R$ is called {\em plurisubharmonic}\footnote{There is a more general notion of {\em plurisubharmonic functions}; for our purpose, we only consider this restrictive definition.} if for any homomorphic map $\phi : V( \subset\C) \ra X$, the composition $f\circ \phi$ is subharmonic. 
Plurisubharmonic functions $f$ satisfy the maximum principle; 
in particular, if $f$ restricts to a nonconstant function on a connected complex subvariety $Y\subset X$, then $Y$ is noncompact. 
 
Now we turn to our construction of plurisubharmonic functions.
Let $\Ga< \Aut(\B^n)$ be a discrete subgroup.
Consider the Poincar\'e series
\begin{equation}\label{eqn:series}
 \sum_{\gamma\in\Gamma} ({1 - |\gamma(z)|^2}), \quad z\in\B^n.
\end{equation}

\begin{lemma}\label{prop:convergence}
Suppose that $\delta(\Gamma)<2$. Then (\ref{eqn:series}) uniformly converges on compact sets.
\end{lemma}

\begin{proof}
Since $\delta(\Gamma) < 2$, the Poincar\'e series
\[
 \sum_{\gamma\in\Gamma} e^{ -2 d(0,\gamma(z))}
\]
uniformly converges on compact subsets in $\B^n$. By (\ref{eqn:dist}),
we get
\begin{equation}\label{eqn:bound}
 e^{-2d(0,\gamma(z))} \le (1 - |\gamma(z)|^2) \le 4 e^{-2d(0,\gamma(z))}.
\end{equation}
Then, the result follows from the upper inequality.
\end{proof}

\begin{remark}\label{rem:series_diverge}
Note that when $\delta(\Gamma)>2$, or when $\Gamma$ is of {\em divergent type} (e.g., convex-cocompact) and $\delta(\Gamma) = 2$, then (\ref{eqn:series}) does not converge. This follows from the lower inequality of (\ref{eqn:bound}).
\end{remark}

Assume that $\delta(\Gamma)<2$. Define $F: \B^n\rightarrow \mathbb{R}$,
\[
F(z) = \sum_{\gamma\in\Gamma} ({ |\gamma(z)|^2 - 1}).
\]
Since $F$ is $\Gamma$-invariant, i.e., $F(\gamma z) = F(z)$, for all $\gamma\in\Gamma$ and all $z\in\B^n$, $F$ descends  to a function $$f : M_\Ga  \rightarrow \mathbb{R}.$$

\begin{lemma}\label{prop:psh}
The function $f : M_\Ga \rightarrow \mathbb{R}$ is plurisubharmonic.
\end{lemma}

\begin{proof}
Enumerate $\Gamma$ as $\Gamma = \{\gamma_1,\gamma_2,\dots\}$.
Consider the sequence of partial sums of the series $F$,
\[
S_k(z) = \sum_{j\le k} ({ |\gamma_j(z)|^2 - 1}).
\]
Since each summand in the above is plurisubharmonic\footnote{This follows from the fact that the function $|z|^2$ is plurisubharmonic.}, $S_k$ is plurisubharmonic for each $k\ge 1$.
Moreover, the sequence of functions $S_k$ is monotonically decreasing. 
Thus, the limit $F = \lim_{k\ra\infty} S_k$ is also plurisubharmonic, and hence so is $f$.
\end{proof}

Note, however, that at this point we do not yet know that the function $f$ is nonconstant. 

\medskip 
Now we prove the main result of this section.

\begin{theorem}\label{thm:no_subv}
Let $\Gamma$ be a torsion-free
discrete subgroup of $\Aut(\B^n)$. If $\delta(\Gamma) < 2$, then
$M_\Ga$ contains no compact complex subvarieties of positive dimension.
\end{theorem}
\begin{proof} Suppose that $Y$ is a compact connected subvariety of positive dimension in $M_\Ga$. Since $\pi_1(Y)$ is finitely generated, so is its image $\Ga'$ in $\Ga=\pi_1(M_\Ga)$. Since $\delta(\Ga')\le \delta(\Ga)$, by passing to the subgroup $\Ga'$ we can (and will) assume that the group $\Ga$ is finitely generated.

We construct a sequence of functions $F_k: \B^n\ra \R$ as follows. For $k \in \N$, let $\Sigma_k\subset \Ga-\{1\}$ denote the subset consisting of $\ga\in \Ga$  
satisfying $d(0,\gamma(0)) \le k$. Since $\Ga$ is a finitely generated linear group, it is residually finite and, hence, there exists a finite index subgroup $\Ga_k<\Ga$ disjoint from $\Sigma_k$. 
For each $k\in \N$, define $F_k: \B^n\ra \R$ as the sum 
\[
F_k(z) = \sum_{\gamma\in\Gamma_k} ({ |\gamma(z)|^2 - 1}).
\]
Since
$$
\bigcap_{k\in \N} \Ga_k=\{1\}, 
$$
the sequence of functions $F_k$ converges to $(|z|^2-1)$ uniformly  on compact subsets of $\B^n$. As before, each $F_k$ is 
plurisubharmonic
(cf. Lemmata \ref{prop:convergence}, \ref{prop:psh}).

Let $\widetilde Y$ be a connected component of  the preimage of $Y$ under the projection map $\B^n \ra M_\Ga$.
Since $\widetilde Y$ is a closed, noncompact subset of $\B^n$, the function $(|z|^2-1)$ is nonconstant on $\widetilde Y$.
As the sequence $(F_k)$ converges to $(|z|^2-1)$ uniformly on compacts, there exists $k\in \N$ such that $F_k$ is nonconstant on $\widetilde Y$. 
Let $f_k : M_k = M_{\Gamma_k} \ra \R$ denote the function obtained by projecting $F_k$  to $M_k$, and $Y_k$ be the image of $\widetilde Y$ under the projection map $\B^n\ra M_k$.
Since $M_k$ is a finite covering of $M_\Ga$, the subvariety $Y_k\subset M_k$ is compact.
Moreover, $f_k$ is a nonconstant plurisubharmonic function on $Y_k$ since $F_k$ is such a function on $\widetilde Y$.
This contradicts the maximum principle.
\end{proof}

\begin{remark}\label{rem:no_subv}
Regarding Remark \ref{rem:series_diverge}: The failure of convergence of the series (\ref{eqn:series}) as pointed out in Remark \ref{rem:series_diverge} is not so surprising. In fact, if $\Gamma$ is a complex Fuchsian group, then $\delta(\Gamma) = 2$ and the convex core of $M_\Ga$ is a compact Riemann surface, see Example \ref{ex:cFuch}.
 Thus, our construction of $F$ must fail in this case. 
\end{remark}

We conclude this section with a proof of the main result of this paper.

\begin{proof}[Proof of Theorem \ref{thm:stein}]
By Theorem \ref{thm:no_subv}, $M_\Ga$ does not have compact complex subvarieties of positive dimensions.
Then, by the second part of Proposition \ref{prop:stein}, $M_\Ga$ is Stein.
\end{proof}

\section{Further remarks}\label{sec:remarks}

In relation to Theorem \ref{thm:stein}, it is also interesting to understand the case when $\delta(\Gamma) = 2$, that is: 
For which convex-cocompact, torsion-free subgroups $\Ga$ of $\Aut(\B^n)$ satisfying $\delta(\Ga)=2$, is the manifold $M_\Ga$ Stein?
It has been pointed out before that a complex Fuchsian subgroup $\Ga< \Aut(\B^n)$ satisfies $\delta(\Ga) = 2$, but the manifold $M_\Ga$ is not Stein. In fact, the convex core of $M_\Ga$ is a complex curve, see Remark \ref{rem:no_subv}.
We conjecture that complex Fuchsian subgroups are the only such non-Stein examples.

\begin{conjecture}
Let $\Gamma<\Aut(\B^n)$ be a convex-cocompact, torsion-free subgroup such that $\delta(\Gamma) = 2$. Then, $M_\Ga$ is non-Stein if and only if $\Gamma$ is a complex Fuchsian subgroup.
\end{conjecture}

We illustrate this conjecture in the following very special case: Let $\phi: \pi_1(\Sigma) \ra \Aut(\B^n)$ be a faithful convex-cocompact representation where $\Sigma$ is a compact Riemann surface of genus $g \ge 2$.
Then $\phi$ induces a (unique) equivariant harmonic map 
\[
F: \widetilde\Sigma \ra \B^n. 
\]
which descends to a harmonic map $f: \Sigma\to M_\Ga$.

\begin{proposition}
Suppose that $F$ is a holomorphic immersion. Then $\Gamma = \phi(\pi_1(\Sigma))$ satisfies $\delta(\Gamma)\ge 2$. Moreover, if $\delta(\Gamma)= 2$, then $\Gamma$ preserves a complex line. In particular, $\Gamma$ is a complex Fuchsian subgroup of $\Aut(\B^n)$.
\end{proposition}

\begin{proof}
Noting that $M_\Ga$ contains a compact complex curve, namely $f(\Sigma)$, the first part follows directly from Theorem \ref{thm:stein}.

For the second part, we let $Y$ denote the surface $\tilde\Sigma$ equipped with the Riemannian metric obtained via pull-back of the Riemannian metric $g$ on $\B^n$.  The entropy\footnote{The {\em volume entropy} of a simply connected Riemannian manifold $(X,g)$ is defined as $\lim_{r\ra\infty} \log\vol(B(r,x))/r$, where $x\in X$ is a chosen base-point and $B(r,x)$ denotes the ball of radius $r$ centered at $x$. This limit exists and is independent of $x$, see \cite{Manning}.} $h(Y)$ of $Y$ is bounded above by $\delta(\Ga)$, i.e.
\begin{equation}\label{ineq:entropy}
 h(Y) \le 2.
 \end{equation}
This can be seen as follows: 
The distance function $d_Y$ on $Y$ satisfies $$d_Y(y_1,y_2) \ge d(F(y_1), F(y_2)).$$
Therefore, the exponential growth-rate $\delta_Y$ of $\pi_1(\Sigma)$-orbits in $Y$ satisfies $\delta_Y\le \delta(\Ga)$. On the other hand, the quantity $\delta_Y=h(Y)$ since $\pi_1(\Sigma)$ acts cocompactly on $Y$. 

Assume that $\widetilde\Sigma$ is endowed with a conformal Riemannian metric of constant   $-4$ sectional curvature.
Since $\tilde\Sigma$ is  a symmetric space, we have
\[
h^2(Y)\mathrm{Area}(Y/\Gamma) \ge h^2(\widetilde\Sigma)\mathrm{Area}(\Sigma),
\]
see \cite[p. 624]{BCG}.
The inequality (\ref{ineq:entropy}) together with the above implies that $\mathrm{Area}(Y/\Gamma) \ge \mathrm{Area}(\Sigma)$.

On the other hand, since $f: Y/\Gamma\to M_\Ga$ is holomorphic, $4\cdot \mathrm{Area}(Y/\Gamma)$ equals to the Toledo invariant $c(\phi)$ (see \cite{Toledo}) of the representation $\phi$. Since $c(\phi) \le 4\pi(g-1)$, the inequality $\mathrm{Area}(Y/\Gamma) \ge \mathrm{Area}(\Sigma) = \pi(g-1)$ shows that $\mathrm{Area}(Y/\Gamma) = \pi(g-1)$ or, equivalently, $c(\phi) = 4\pi(g-1)$.
By the main result of \cite{Toledo}, $\Gamma$ preserves a complex-hyperbolic line in $\B^n$.
\end{proof}

\begin{remark}
The assumption that $F$ is an immersion can be eliminated: Instead of working with a Riemannian metric, one can work with a  Riemannian metric with finitely many singularities. 
\end{remark}

Motivated by Theorem \ref{thm:no_subv}, we also make the following conjecture.

\begin{conjecture}
 If $\Ga< \Aut(\B^n)$ is discrete, torsion-free, and $\delta(\Ga)<2k$, then $M_\Ga$ does not contain compact complex subvarieties of dimension $\ge k$. 
\end{conjecture}

We conclude this section with a verification of this conjecture under a stronger hypothesis.

\begin{proposition}\label{prop:2k}
  If $\Ga< \Aut(\B^n)$ is discrete,  torsion-free, and $\delta(\Ga)<2k-1$, then $M_\Ga$ does not contain compact complex subvarieties of dimension $\ge k$. 
\end{proposition}

\begin{proof}
Note that if $\Gamma$ is elementary (i.e., virtually abelian), then $\delta(\Gamma) = 0$. In this case, the result follows from Theorem \ref{thm:no_subv}.
For the rest, we assume that $\Gamma$ is nonelementary.

 By \cite[Sec. 4]{BCG2}, there is a natural map $f: M_\Ga\ra M_\Ga$ homotopic to the identity map $\mathrm{id}_{M_\Ga}: M_\Ga\ra M_\Ga$ and satisfying
 \[
 |\mathrm{Jac}_p(f)| \le \left(\frac{\delta(\Ga)+1}{p}\right)^p,\quad 2\le p\le 2n,
 \]
 where $\mathrm{Jac}_p(f)$ denotes the $p$-Jacobian of $f$.
 When $\delta(\Ga)<2k-1$, we have $|\mathrm{Jac}_p(f)|<1$, for  $p\in[2k,2n]$.
 This means that $f$ strictly contracts the volume form on each $p$-dimensional tangent space at every point $x\in M_\Ga$, for $p\in [2k,2n]$.
 
 Let $Y\subset M_\Ga$ be a compact complex subvariety of dimension $\ge k$ (real dimension $\ge 2k$).
 Then, $Y$ is also a volume minimizer in its homology class.
 Since $f$ strictly contracts volume on $Y$, $f(Y)$ has volume strictly lesser than that of $Y$.
 However, $f$ being homotopic to $\mathrm{id}_{M_\Ga}$, $f(Y)$  belongs to the homology class of $Y$.
 This is a contradiction to the fact that $Y$ minimizes volume its homology class.
\end{proof}

\begin{remark}
 Note that Proposition \ref{prop:2k} gives an alternative proof of Theorem \ref{thm:no_subv} (hence Theorem \ref{thm:stein}) under a stronger hypothesis, namely $\delta(\Gamma)\in(0,1)$.
 However, this method fails to verify Theorem \ref{thm:no_subv} in the case when $\delta(\Gamma)\in[1,2)$.
\end{remark}

Finally, we note that the papers \cite{Chen} and \cite{Yue} contain  other interesting 
results and conjectures on Stein properties of complex-hyperbolic manifolds.

\bibliographystyle{abbrv}
%\bibliography{bibliography.bib}

\end{document}